\documentclass{amsart}

\usepackage{amsfonts, amsmath, amscd}
\usepackage[psamsfonts]{amssymb}

\usepackage{amssymb}

\usepackage{pb-diagram}

\usepackage[all,cmtip]{xy}

\usepackage[usenames]{color}

\newtheorem{theorem}{Theorem}[section]
\newtheorem{lemma}[theorem]{Lemma}

\newtheorem{remark}[theorem]{Remark}

\newtheorem{definition}[theorem]{Definition}

\numberwithin{equation}{section}

\newcommand{\NN}{\mathbb{N}}
\newcommand{\ZZ}{\mathbb{Z}}

\newcommand{\w}{\omega}

\newcommand{\TTT}{\mathcal{T}}

\newcommand{\IR}{\mathbb{R}}

\newcommand{\Ss}{\mathbb{S}}

\newcommand{\Ff}{\mathfrak{F}}

\renewcommand{\phi}{\varphi}

\input xy
\xyoption{all}

\title{A locally quasi-convex abelian group without  Mackey topology}

\author[S.~Gabriyelyan]{Saak Gabriyelyan}
\address{Department of Mathematics, Ben-Gurion University of the
Negev, Beer-Sheva, P.O. 653, Israel}
\email{saak@math.bgu.ac.il}



\begin{document}

\begin{abstract}
We give the first example of a locally quasi-convex (even countable reflexive and $k_\w$) abelian group $G$ which does not admit the strongest compatible locally quasi-convex group topology. Our group $G$ is the Graev free abelian group $A_G(\mathbf{s})$ over a convergent sequence $\mathbf{s}$.
\end{abstract}

\maketitle


\section{Introduction}


Let $(E,\tau)$ be a locally convex space. A locally convex vector topology $\nu$ on $E$ is called {\em compatible with $\tau$} if the spaces $(E,\tau)$ and $(E,\nu)$ have the same topological dual space. The famous Mackey--Arens theorem states the following
\begin{theorem}[Mackey--Arens] \label{t:Mackey-M-A}
Let $(E,\tau)$ be a locally convex space. Then $(E,\tau)$ is a {\em pre-Mackey} locally convex space in the sense that there is the finest locally convex vector space topology $\mu$ on $E$ compatible with $\tau$. Moreover, the topology $\mu$ is the topology of uniform convergence on absolutely convex weakly* compact subsets of the topological dual space $E'$ of $E$.
\end{theorem}
The topology $\mu$ is called the {\em Mackey topology} on $E$ associated with $\tau$, and if $\mu=\tau$, the space $E$ is called a {\em Mackey space}.

For an abelian topological group $(G,\tau)$ we denote by $\widehat{G}$ the group of all continuous characters of $(G,\tau)$.
Two topologies  $\mu$ and $\nu$ on an abelian group $G$  are said to be {\em compatible } if $\widehat{(G,\mu)}=\widehat{(G,\nu)}$.
Being motivated by the Mackey--Arens Theorem  \ref{t:Mackey-M-A} the following notion was  introduced and studied in \cite{CMPT} (for all relevant definitions see the next section):
\begin{definition}[\cite{CMPT}]{\em
A locally quasi-convex  abelian group $(G,\mu)$ is called a {\em Mackey group} if for every locally quasi-convex group topology $\nu$ on $G$ compatible with $\tau$  it follows that $\nu\leq\mu$. In this case the topology $\mu$ is called a {\em Mackey topology} on $G$. A locally quasi-convex abelian group $(G,\tau)$ is called a  {\em pre-Mackey group} and $\tau$ is called a {\em pre-Mackey topology on $G$} if there is a Mackey topology $\mu$ on $G$  associated with $\tau$. }
\end{definition}
Not every Mackey locally convex space is a Mackey group. Indeed, answering a question posed in \cite{DMPT}, we proved in \cite{Gab-Mackey} that the metrizable locally convex space $(\IR^{(\NN)},\mathfrak{p}_0)$ of all finite sequences with the topology $\mathfrak{p}_0$ induced from the product space $\IR^{\NN}$ is not a Mackey group. In \cite{Gab-Cp} we show that the space $C_p(X)$, which is a Mackey space for every Tychonoff space $X$, is a Mackey group if and only it is barrelled.

A weaker notion than to be a Mackey group was introduced in \cite{Gab-Mackey}. Let $(G,\tau)$ be a locally quasi-convex abelian group. A locally quasi-convex group topology $\mu$ on $G$ is called {\em quasi-Mackey} if $\mu$ is compatible with $\tau$ and there is no locally quasi-convex group topology  $\nu$ on $G$ compatible with $\tau$ such that $\mu<\nu$. The group $(G,\tau)$  is {\em quasi-Mackey} if $\tau$ is a quasi-Mackey topology. Proposition 2.8 of \cite{Gab-Mackey} implies that every locally quasi-convex abelian group has quasi-Mackey topologies. 

The Mackey--Arens theorem suggests the following general  question posed in \cite{CMPT}: {\em Does every locally quasi-convex  abelian group is a pre-Mackey  group}? In the main result of the paper, Theorem \ref{t:A(s)-Mackey}, we answer this question  in the negative.

Let $\mathbf{s}=\{ 0\} \cup\{1/n: n\in\NN\}$ be the convergent sequence endowed with the topology induced from $\IR$. Denote by $A_G(\mathbf{s})$ the (Graev) free abelian group over $\mathbf{s}$. Note that the group $A_G(\mathbf{s})$ is a countable reflexive and $k_\w$-group, see \cite{Gab} and \cite{Gra} respectively.
In Question 4.4 of \cite{Gab-Mackey} we ask: {\em Is it true that $A_G(\mathbf{s})$ is a Mackey group}? Below
 we answer this question negatively in a stronger form.
\begin{theorem} \label{t:A(s)-Mackey}
The  group $A_G(\mathbf{s})$ is neither a pre-Mackey group nor a quasi-Mackey group.
\end{theorem}
This result gives the first example of a locally quasi-convex group which is not pre-Mackey additionally showing a big difference between the case of locally quasi-convex groups and the case of locally convex spaces.


\section{Proof of Theorem \ref{t:A(s)-Mackey}} \label{sec:Mackey-F0}



Set $\NN:=\{ 1,2,\dots\}$.
Denote by $\mathbb{S}$ the unit circle group and set $\Ss_+ :=\{z\in  \Ss:\ {\rm Re}(z)\geq 0\}$.
Let $G$ be an abelian topological group.  If $\chi\in \widehat{G}$, it is considered as a homomorphism from $G$ into $\mathbb{S}$.
A subset $A$ of $G$ is called {\em quasi-convex} if for every $g\in G\setminus A$ there exists   $\chi\in \widehat{G}$ such that $\chi(x)\notin \Ss_+$ and $\chi(A)\subseteq \Ss_+$.
An abelian topological group $G$ is called {\em locally quasi-convex} if it admits a neighborhood base at the neutral element $0$ consisting of quasi-convex sets. It is well known that the class of locally quasi-convex abelian groups is closed under taking products and subgroups. The dual group $\widehat{G}$ of $G$ endowed with the compact-open topology is denoted by $G^{\wedge}$. The homomorphism $\alpha_G : G\to G^{\wedge\wedge} $, $g\mapsto (\chi\mapsto \chi(g))$, is called {\em the canonical homomorphism}. If $\alpha_G$ is a topological isomorphism the group $G$ is called {\em  reflexive}. Any reflexive  group is locally quasi-convex.

Let $X$ be a Tychonoff space with a distinguished point $e$. Following \cite{Gra}, an abelian topological group $A_G(X)$ is called {\em  the Graev free abelian topological  group} over  $X$ if $A_G(X)$ satisfies the following conditions:
\begin{enumerate}
\item[{\rm (i)}] $X$ is  a subspace of $A_G(X)$;
\item[{\rm (ii)}] any continuous map $f$ from $X$ into any abelian topological group $H$, sending $e$ to the identity of $H$, extends uniquely to a continuous homomorphism ${\bar f}: A_G(X) \to H$.
\end{enumerate}
For every Tychonoff space $X$, the Graev free abelian topological  group $A_G(X)$ exists, is unique up to isomorphism of abelian topological groups,  and is independent of the choice of $e$ in $X$, see \cite{Gra}. Further, $A_G(X)$ is algebraically the free abelian group on $X\setminus \{ e\}$.

We denote by $\tau$ the topology of the group $A_G(\mathbf{s})$. For every $n\in\NN$, set
\[
e_n :=(0,\dots,0,1,0,\dots)\in \ZZ^{(\NN)},
\]
where $1$ is placed in position $n$ and $\ZZ^{(\NN)}$ is the direct sum $\bigoplus_\mathbb{N} \mathbb{Z}$. Now the map $i(1/n):=e_n$, $n\in\NN$, defines an algebraic isomorphism of  $A_G(\mathbf{s})$ onto $\ZZ^{(\NN)}$. So we can identify  algebraically $A_G(\mathbf{s})$ and $\ZZ^{(\NN)}$.

Let $g_n$ be a sequence in $A_G(\mathbf{s})$ of the form
\[
g_n =(0,\dots, 0, r^n_{i_n},  r^n_{i_n+1}, r^n_{i_n+2}, \dots ),
\]
where $i_n\to \infty$ and there is a $C>0$ such that $\sum_j |r^n_j| \leq C$ for every $n\in\NN$. Since $e_n\to 0$ in $\tau$ we obtain
\begin{equation} \label{equ:Mackey-Free-0}
g_n \to 0 \quad \mbox{ in } \tau.
\end{equation}

The following group plays an essential role in the proof of Theorem \ref{t:A(s)-Mackey}.
Set
\[
c_0 (\Ss):= \{ (z_n) \in \Ss^\mathbb{N} :\; z_n \to 1\},
\]
and denote by $\Ff_0 (\Ss)$ the group $c_0 (\Ss)$ endowed with the metric  $d\big((z_n^1), (z_n^2 ) \big)= \sup \{ |z_n^1 -z_n^2 |, n\in\NN \}$. Then $\Ff_0 (\Ss)$ is a Polish group, and the sets of the form $V^\NN \cap c_0(\Ss)$, where $V$ is an open neighborhood at the identity $1$ of $\Ss$, form a base at $1$ in $\Ff_0 (\Ss)$. Actually $\Ff_0 (\Ss)$ is isomorphic to $c_0/\ZZ^{(\NN)}$  (see \cite{Gab}). In  \cite{Gab}  we proved that the group $\Ff_0 (\Ss)$ is reflexive and $\Ff_0 (\Ss)^\wedge = A_G(\mathbf{s})$.

If $g$ is an element of an abelian group $G$, we denote by $\langle g\rangle$ the subgroup of $G$ generated by $g$.
We need the following lemma.
\begin{lemma} \label{l:lemma-F0-Mackey}
Let $z,w\in\Ss$ and let $z$ have infinite order. Let $V$ be a neighborhood of $1$ in $\Ss$. If $w^l=1$ for every $l\in\NN$ such that $z^l\in V$, then $w=1$.
\end{lemma}

\begin{proof}
The main result of \cite{BDS} applied to $\langle z\rangle$ states the following: there exists a sequence $A=\{ a_n\}_{n\in\NN}$ in $\NN$ such that if $v\in \Ss$, then
\[
\lim_n v^{a_n} =1 \; \mbox{ if and only if } \; v\in \langle z \rangle.
\]
Now suppose for a contradiction that $w\not=1$. Since $\langle z\rangle$ is dense in $\Ss$, there is an $l\in\NN$ such that $z^l\in V$. So $w$ has finite order, say $q$. Observe that $w\not\in \langle z \rangle$. Then, by assumption,  for every $l\in\NN$ such that $z^l\in V$ we have $w^l=1$, and hence there is a $c(l)\in \NN$ such that $l=c(l)\cdot q$. Since $\lim_n z^{a_n} =1$, there exists an $N\in\NN$ such that $z^{a_n} \in V$ for every $n>N$. So $a_n =c(a_n)\cdot q$ for every $n>N$. But in this case we trivially have  $\lim_n w^{a_n} =1$ which contradicts the choice of the sequence $A$ since $w\not\in \langle z \rangle$. Thus $w=1$.
\end{proof}

In the proof of Theorem \ref{t:A(s)-Mackey} we use the following result, see Theorem 2.7 of  \cite{Gab-Mackey}.
\begin{theorem}[\cite{Gab-Mackey}] \label{t:Char-Mackey}
For a locally quasi-convex abelian group $(G,\tau)$ the following assertions are equivalent:
\begin{enumerate}
\item[{\rm (i)}] the group  $(G,\tau)$ is  pre-Mackey;
\item[{\rm (ii)}] $\tau_1\vee\tau_2$ is compatible with $\tau$ for every locally quasi-convex group topologies $\tau_1$ and $\tau_2$ on $G$ compatible with $\tau$.
\end{enumerate}
\end{theorem}

Now we are ready to prove Theorem \ref{t:A(s)-Mackey}.
\begin{proof}[Proof of Theorem \ref{t:A(s)-Mackey}]
First we construct a family 
\[
\{ \TTT_z : z\in\Ss \mbox{ has infinite order}\}
\]
of topologies on $\ZZ^{(\NN)}$ compatible with the topology $\tau$ of $A_G(\mathbf{s})$. To this end, we use the idea described in Proposition 4.1 of \cite{Gab-Mackey}.

Let $z\in\Ss$ be of infinite order. For every $i\in\NN$, set
\[
\chi_i :=(0,\dots,0,z,0,\dots) \in \Ff_0(\Ss)=A_G(\mathbf{s})^\wedge,
\]
where $z$ is placed in position $i$. For every $(n_k)\in A_G(\mathbf{s})$, it is clear that $\chi_i\big( (n_k)\big)=1$ for all sufficiently large $i\in\NN$ (i.e., $\chi_i\to 1$ in the pointwise topology on $\Ff_0(\Ss)$). So  we can define the following algebraic monomorphism $T_z:\ZZ^{(\NN)}\to A_G(\mathbf{s})\times \Ff_0(\Ss)$ by
\begin{equation} \label{equ:Mackey-Free-1}
T_z\big( (n_k)\big):= \bigg( (n_k), \big( \chi_i\big((n_k)\big)\big)\bigg)=\bigg( (n_k), \big( z^{n_k}\big) \bigg), \quad \forall \; (n_k)\in \ZZ^{(\NN)}.
\end{equation}
Denote by $\TTT_z$ the locally quasi-convex topology on $\ZZ^{(\NN)}$ induced from $A_G(\mathbf{s})\times \Ff_0(\Ss)$.

{\em Claim 1. The topology $\TTT_z$ is compatible with $\tau$.}  Indeed, set $G:= \big(\ZZ^{(\NN)}, \TTT_z\big)$. Since $\TTT_z$ is weaker than the discrete topology $\tau_d$ on $\ZZ^{(\NN)}$, we obtain $G^\wedge \subseteq \big( \ZZ^{(\NN)}, \tau_d\big)^\wedge= \Ss^\NN$. Fix arbitrarily $\chi=(y_n)\in G^\wedge$. To prove the claim we have to show that $y_n \to 1$.

Suppose for a contradiction that $y_n\not\to 1$. As $\Ss$ is compact we can find a sequence $0< m_1< m_2<\dots$ of indices such that $y_{m_i} \to w \not= 1$ at $i\to\infty$. Since $\chi$ is $\TTT_z$-continuous,  there exists a standard neighborhood $W=T_z^{-1}\big(U\times V^\NN\big)$ of zero in $G$, where $U$ is a $\tau$-neighborhood of zero in $A_G(\mathbf{s})$ and $V$ is a neighborhood of $1$ in $\Ss$, such that $\chi(W)\subseteq \Ss_+$. Observe that, by (\ref{equ:Mackey-Free-1}),  $(n_k)\in W$ if and only if
\begin{equation} \label{equ:Mackey-Free-11}
(n_k)\in U \mbox{ and } z^{n_k} \in V \mbox{ for every } k\in \NN,
\end{equation}
and, the inclusion $\chi(W)\subseteq \Ss_+$ means that
\begin{equation} \label{equ:Mackey-Free-2}
\chi\big( (n_k)\big) =\prod_k y_k^{n_k} \in \Ss_+, \; \mbox{ for every } (n_k)\in W.
\end{equation}
We assume additionally that $w\not\in V$.  Since $\langle z\rangle$ is dense in $\Ss$, choose arbitrarily an $l\in\NN$ such that $z^l\in V$. Fix arbitrarily a $t\in\NN$.  Now, by (\ref{equ:Mackey-Free-0}), there is an $N(t)\in \NN$ such that every $\mathrm{x}_{it}:=(n_k)\in \ZZ^{(\NN)}$ of the form
\begin{equation} \label{equ:Mackey-Free-3}
\mathrm{x}_{it}= (0,\dots,0, \underbrace{l}_{m_{i+1}},0,\dots,0,\underbrace{l}_{m_{i+2}},0,\dots,0,\underbrace{l}_{m_{i+t}},0,\dots)
\end{equation}
belongs to $W$ for every $i\geq N(t)$. For every  $\mathrm{x}_{it}\in W$ of the form (\ref{equ:Mackey-Free-3}), (\ref{equ:Mackey-Free-2}) implies
\begin{equation} \label{equ:Mackey-Free-4}
\chi\big( \mathrm{x}_{it} \big) = \big( y_{m_{i+1}} \cdots y_{m_{i+t}}\big)^l \to w^{lt}, \quad \mbox{ at } i\to\infty.
\end{equation}
Now, if $w^l \not=1$ for some $l\in\NN$ such that $z^l\in V$, then there exists a $t\in\NN$ such that $w^{lt}\not\in \Ss_+$. Therefore, by (\ref{equ:Mackey-Free-4}),  $\chi(W)\nsubseteq \Ss_+$, a contradiction. Assume that  $w^l =1$ for every $l\in\NN$ such that $z^l\in V$. Then Lemma \ref{l:lemma-F0-Mackey} implies $w=1$ that is impossible. So our assumption that $y_n\not\to 1$ is wrong. Therefore $y_n\to 1$ and $\widehat{G}=c_0(\Ss)$. Thus $\TTT_z$ is compatible with $\tau$.

{\em Claim 2. For every element $a\in\Ss$ of finite order, the topology $\TTT_z \vee \TTT_{az}$ is not compatible with $\tau$.} Let $r$ be the order of $a$. Consider standard neighborhoods
\[
W_z=T_z^{-1}\big(U\times V^\NN\big) \in \TTT_z \;\; \mbox{ and } \;\; W_{az}=T_{az}^{-1}\big(U\times V^\NN\big) \in\TTT_{az},
\]
where $U\in\tau$ and a symmetric neighborhood $V$ of $1$ in $\Ss$ is chosen such that $V\cdot V\cap \langle a\rangle =\{ 1\}$. Then, by (\ref{equ:Mackey-Free-11}), we have
\[
W_z \cap W_{az} = \left\{ (n_k)\in \ZZ^{(\NN)} : (n_k)\in U \mbox{ and } z^{n_k}, (az)^{n_k} \in V \mbox{ for every } k\in\NN\right\}.
\]
In particular, $a^{n_k}\in V\cdot V$, and hence $a^{n_k}=1$ for every $k\in\NN$. Therefore, for every $k\in\NN$, there is an $s_k\in\NN$ such that $n_k = s_k \cdot r$. Set $\eta:=(a,a,\dots)\in \Ss^\NN$. Then $\eta(W_z \cap W_{az} )=\{ 1\}$. As $W_z \cap W_{az}\in \TTT_z \vee \TTT_{az}$ it follows that $\eta$ is $\TTT_z \vee \TTT_{az}$-continuous. Since $\eta\not\in c_0(\Ss)$ we obtain that $\TTT_z \vee \TTT_{az}$ is not compatible with $\tau$.

{\em Claim 3. $\tau <\TTT_z$, so $\tau$ is not quasi-Mackey.} By (\ref{equ:Mackey-Free-1}), it is clear that $\tau \leq\TTT_z$. To show that $\tau\not= \TTT_z$, suppose for a contradiction that $\TTT_z =\tau$. Then, by Claim 1, $\TTT_z \vee \TTT_{az}=\tau \vee \TTT_{az}= \TTT_{az}$ is compatible with $\tau$. But this contradicts Claim 2.

{\em Claim 4. The group $A_G(\mathbf{s})$ is not pre-Mackey}. This immediately follows from Claim 2  and Theorem \ref{t:Char-Mackey}.
\end{proof}

\begin{remark} {\em
Just before submission of the preprint, Prof. Lydia Au{\ss}enhofer informed the author that she had also solved the problem posed by me: namely if the group $A_G(\mathbf{s})$ is a Mackey group and proved Theorem \ref{t:A(s)-Mackey}, see \cite{Aus3}. It  is worth mentioning that my proof totally differs from hers, being much simpler and shorter. }
\end{remark}


\bibliographystyle{amsplain}

\end{document}